\documentclass[11pt,a4paper]{amsart}
\usepackage{amsmath,amsfonts,amssymb,amsthm,epsfig,epstopdf,url,array}
\usepackage[english]{babel}

\graphicspath{ {art/} }

\usepackage{enumitem}
\usepackage{float}
\usepackage{MnSymbol}
\usepackage{lipsum}
\usepackage{amsthm}
\usepackage{zref}
\usepackage{cleveref}
\usepackage{mathtools}

\DeclareMathOperator{\sdim}{s-dim}
\newcommand{\sphere[1]}{\circ_{<}^{#1}}

\newcommand{\subgraph}{\subseteq}
\newenvironment{pf}[1]
    {
    \vspace{0.5em}
    \noindent
    \emph{Proof of #1.}}
    {
    \qed
    \break
    }
\newcommand{\suchthat}{\;\ifnum\currentgrouptype=16 \middle\fi|\;}

\theoremstyle{plain}
\newtheorem{thm}{Theorem}[section] 
\newtheorem{lem}[thm]{Lemma}
\newtheorem{prop}[thm]{Proposition}
\newtheorem{cor}[thm]{Corollary}

\theoremstyle{definition}

\theoremstyle{remark}

\title{Graphs That Are Minor Minimal with Respect to Dimension}
\author{Thomas Giardina and Joel Foisy}
\thanks{Keywords: embedded graph, graph dimension. \\
\indent 2010 \textit{Mathematics Subject Classification.} 05C10, 05C83.}
\renewcommand\footnotemark{}

\begin{document}

\maketitle


\begin{abstract}
    Erd\H{o}s, Harary, and Tutte defined the dimension of a graph $G$ as the smallest natural number $n$ such that $G$ can be embedded in $\mathbb{R}^n$ with each edge a straight line segment of length 1. 
    Since the proposal of this definition, little has been published on how to compute the exact dimension of graphs and almost nothing has been published on graphs that are minor minimal with respect to dimension. This paper develops both of these areas. In particular, it (1) establishes certain conditions under which computing the dimension of graph sums is easy and (2) constructs three infinitely-large classes of graphs that are minor minimal with respect to their dimension.
\end{abstract}


\section{Introduction}

This paper plays a role in the completion of two goals, both having to do with the classification of graphs. Letting ``dimension" mean dimension as defined in [2], the first goal is to classify all graphs of any given dimension by constructing a complete list of minimal forbidden minors. The three classes of minor minimal graphs presented in this paper provide a partial list of minimal forbidden minors for each dimension. 

The second goal is to find an intuitive geometric notion of dimension that is equivalent to the Colin de Verdi\`ere graph invariant, $\mu(G)$ [1]. Unfortunately, Erd\H{o}s et al's dimension is not equivalent to $\mu$. However, we hope that the examples and techniques developed in this paper, combined with previously published results that tie $\mu(G)$ to $G$'s geometric properties (see \Cref{colin}), will help lead to a definition of dimension that is equivalent to the Colin de Verdi\`ere invariant.
\begin{table}[H]
\centering
\begin{tabular}{ c | c  }
 \textbf{$\mu(G)$} & Geometric Property \\ \hline
$ \leq 0$  & $G$ has no edges [3] \\ \hline
$ \leq 1$  & $G$ is the disjoint union of paths [4] \\ \hline
$\leq 2$  & $G$ is outer-planar [3] \\ \hline
$\leq 3$  & $G$ is planar [3] \\ \hline
$\leq 4$  & $G$ is linklessly embeddable [6]  \\ \hline
 \end{tabular}
 \caption{Known relations between $\mu(G)$ and $G$'s geometric properties.}
 \label{colin}
\end{table}

More specifically, Erd\H{o}s et al defined the \emph{dimension} of a graph, denoted $\dim G$, to be the smallest natural number $n$ such that $G$ can be embedded in $\mathbb{R}^n$ with edges length 1 and vertices mapped to distinct points (edges are allowed to cross one another but not allowed to cross vertices) [2]. Here, we show how to calculate the dimension of the sum of graphs and provide three classes of graphs that are minor minimal with respect to dimension (or in other words, that are minimal forbidden minors for dimension $dim(G)-1$). 

Integral to both of these efforts is a new notion of dimension, \emph{spherical dimension}. The spherical dimension of $G$, denoted $\sdim G$, is defined similarly to the dimension of $G$, save that the vertices must lie on an n-dimensional sphere (where two points is a 1-dimensional sphere, a circle is a 2-dimensional sphere, and so on). By definition, $\dim G \leq \sdim G$. We will see later on that every graph has both a dimension and a spherical dimension (\Cref{exist}).

Throughout, we will use $\sphere[n]$ to mean ``n-dimensional sphere of radius less than 1.''

The main result with respect to the sum of graphs is:
\begin{description}
	\item[\Cref{connector}] Let $|G|, |H| > 1$. If $G$ and $H$ can be embedded on a $\sphere[\sdim G]$ and $\sphere[\sdim H]$ such that the squares of the radii sum to 1, 
\[
\dim(G+H) = \sdim G+\sdim H.
\]
Otherwise, 
\[
\dim(G+H) > \sdim G + \sdim H.
\]
\end{description}
Given a graph $\sum_{i=0}^n G_i$, this theorem often allow us to easily compute $\dim \sum_{i=0}^n G_i$ so long as we have a strong understanding of each $G_i$. This allows us to compute the dimension of large, high-dimensional graphs by through investigating relatively low-dimensional graphs. 

Indeed, this technique allow us to construct the following three classes of minor minimal graphs. Let $\epsilon_n$ be the graph with $n$ vertices and no edges and let 
\begin{center}
$S_1 = \epsilon_1$ \\
$S_2 = \epsilon_2$ \\
$S_3 = \epsilon_3$ \\
$S_4 = K_1 + \epsilon_3$ \\
... \\
$S_n = K_{n-3} + \epsilon_3$.
\end{center}
Then,
\begin{description}
	\item[\Cref{MainK}] $K_n$ is minor minimal with respect to dimension $n-1$.
	\item[\Cref{Sne3}] $S_n + \epsilon_3$ is minor minimal with respect to dimension $n+1$.
	\item[\Cref{final theorem}] For $1 < n < 5$, $S_n + C_6$ is minor minimal with respect to dimension $n+2$, and for $n \geq 5$, $S_n + C_5$ is minor minimal with respect to dimension $n+2$.
\end{description}

Additionally, we find that 
\begin{description}
    \item[\Cref{sBase}, \Cref{Sne3}] $S_n$ is minor minimal with respect to spherical dimension $n-1$.
\end{description}

In the final section, we consider a definition of dimension where edges are not allowed to cross. Most of the proofs in the previous sections still hold using this new definition, and using it will allow us to expand the $S_n + C_6$ class of minor minimal graphs as follows: Let
\[
\mathcal{P}_n = \left \{ P \suchthat \begin{aligned} &\text{$P$ has $n$ edges, and each vertex of $P$} \\
&\text{ is incident to either one or two edges} \end{aligned} \right \}
\]
and
\[
\mathcal{F}_n^G = \{ G + P \suchthat P \in \mathcal{P}_n\}.
\]
Then,
\begin{description}
\item[\Cref{flowers}] The following graphs are minor minimal:
\begin{itemize}
	\item $\mathcal{F}_6^{S_1} - \{ K_1+C_6 \}$ with respect to dimension 3,
	\item $\mathcal{F}_6^{S_n}$ ($1<n<5$) with respect to dimension $n+2$,
	\item $\mathcal{F}_5^{S_n}$ ($n \geq 5$) with respect to dimension $n+2$.
\end{itemize}
\end{description}


\section{Definitions}

We have already seen the two most important definitions, those of dimension and spherical dimension. When considering these definitions, keep in mind that we consider a 1-dimensional sphere to be 2 points, a 2-dimensional sphere to be a circle, and so on. We go against the practice of a 1-sphere being a circle, a 2-sphere being a sphere, and so on, because we wish to emphasize the dimension of the space the object sits in rather than the dimension of the object.

We also saw the phrase ``the sum of graphs'' in the introduction. The \emph{sum of two graphs} $G$ and $H$, denoted $G+H$, is the graph obtained by taking disjoint copies of $G$ and $H$ and adding all possible edges between $G$ and $H$. The sum of more than two graph, $G_1$, $G_2$, ..., $G_n$, is defined iteratively: $(((G_1 + G_2) + G_3)+...+G_n)$.

A minor of a graph is another important notion. A \emph{minor} $H$ of a graph $G$ is a graph that can be obtained by iteratively applying the following three operations to $G$:
\begin{description}
    \item[\hspace{1cm} Vertex Removal] Remove a vertex from $G$.
    \item[\hspace{1cm} Edge Removal] Remove an edge from $G$.
    \item[\hspace{1cm} Edge Contraction] Merge two adjacent vertices so that the new vertex is adjacent to every vertex to which either one of the initial two vertices were adjacent.
\end{description}
A graph $G$ is \emph{minor minimal} with respect to some property if $G$ has the property but none of its proper minors do. In this paper, we are interested in graphs that are minor minimal with respect to some dimension -- i.e., $G$ has some dimension $n$, but all of its proper minors have dimension strictly less than $n$.

For the sake of space, we will also use the following notation:
\begin{description}
    \item[$|G|$] The number of vertices of $G$.
    \item[$G \cupdot H$] The disjoint union of $G$ and $H$.
    \item[$\epsilon_n$] The empty graph on $n$ vertices.
    \item[$C_n$] The cycle graph on $n$ vertices.
    \item[$K_n$] The complete graph on $n$ vertices.
    \item[$H \subgraph G$] $H$ is a (not necessarily proper) subgraph of $G$.
    \item[\texttt{$\sphere[n]$}] An n-dimensional sphere of radius less than 1.
\end{description}
We will sometimes use $\sphere[]$ when we want to discuss a hyper sphere of radius less than 1 and undeclared dimension.

\section{Computing the Dimension of Graphs: The General Wheel Graph}

\begin{figure}[H]
\centering
\includegraphics[width=3cm]{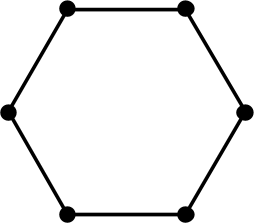} \hspace{0.5cm}
\includegraphics[width=3cm]{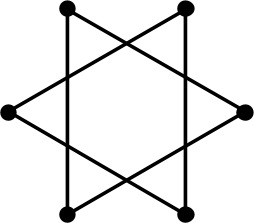}
\hspace{0.5cm}
\includegraphics[width=3cm]{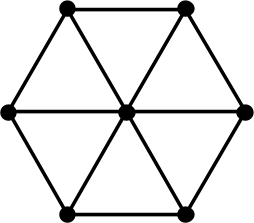}
\caption{A convex hexagon, star hexagon, and $W_6$, respectively. Notice that the convex hexagon is an embedding of $C_6$ but the star hexagon is not.}
\label{hexagons}
\end{figure}

For $n \geq 3$, let $C_n$ denote the cycle graph of length $n$. It is easy to see that $\dim C_n = 2$. A slightly more complicated example is the wheel graph. Define the \emph{wheel graph of length n}, denoted $W_n$, by $W_n = C_n +vertex$. Since $C_n$ is a subgraph of $W_n$,
\[
\dim W_n \geq \dim C_n = 2.
\]
But can we embed $W_n$ in $\mathbb{R}^2$? Suppose we could. Since all the vertices of $C_n$ must be a distance 1 away from $W_n$'s central vertex, $C_n$ must lie on a unit circle. It follows that $C_n$ must be embedded as a regular polygon (either convex or star). However, not all regular n-gons are embeddings of $C_n$. For example the star hexagon in \Cref{hexagons} is not. We call such n-gons \emph{degenerate}. The circum-radius of regular, non-degenerate n-gons is given by
\[
r = \frac{\sin \left ( \frac{n-2m}{n}\pi \right )}{\sin \left (\frac{2m}{n}\pi \right )},
\]
where $m-1$ vertices are skipped when constructing the n-gon[2]. For example, in \Cref{hexagons}, the convex hexagon has $n=6$ and $m=1$ and the star hexagon has $n=6$ and $m=2$ (although, this star hexagon is also degenerate, so the equation does not apply). From this equation we get the following radii:
\begin{table}[H]
\centering
\begin{tabular}{ c | c | c |c }
 \textbf{Radii of Non-Degenerate n-gon} & $m > \frac{1}{6}n$ & $m = \frac{1}{6}n$ & $m < \frac{1}{6}n$ \\ \hline
radius  & $<1$ & 1 & $>1$ 
 \end{tabular}
\caption{The radii of non-degenerate n-gons given $n$ and $m$, the number of vertices skipped by each edge.}
 \label{radii table}
\end{table}
\noindent If $\gcd(n,m) \neq 1$, the n-gon is degenerate, so from the above table, we conclude that no non-degenerate star n-gon has radius 1, and so $C_6$ is the only cyclic graph with an embedding on a unit circle (where the embedding is as a convex hexagon). Therefore, $\dim W_6 = 2$, and $W_n > 2$ for all other $n$.

We now see that all $W_n$ can be embedded in $\mathbb{R}^3$. Similarly to in $\mathbb{R}^2$, $W_n$ can be embedded if and only if all the vertices of $C_n$ can be embedded on a unit sphere. For $3 \leq n<6$, such an embedding is easy; simply place $C_n$ on one of the unit sphere's appropriately-sized lesser circles, as in Figure 3.
\begin{figure}[H]
\centering
\includegraphics[width=3cm]{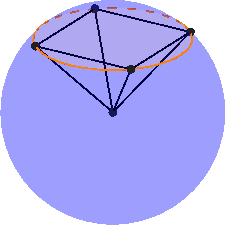}
\caption{An embedding of $W_4$ in $\mathbb{R}^3$}
\end{figure}
For $n \geq 6$, construct the embedding as follows: \begin{figure}[H]
\begin{center}
\includegraphics[width=3cm]{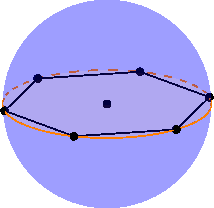}
\hspace{0.5cm}
\includegraphics[width=3cm]{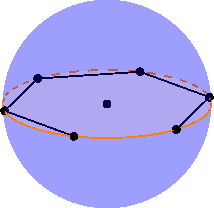}
\hspace{0.5cm}
\includegraphics[width=3cm]{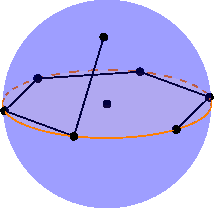} \hspace{0.5cm} \includegraphics[width=3cm]{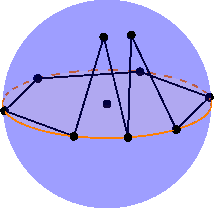} \hspace{0.5cm}
\includegraphics[width=3cm]{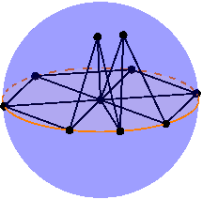}
\end{center}
\caption{The construction of $W_8$ embedded in the unit sphere.}
\label{wheelFig}
\end{figure}
\begin{itemize}
	\item Embed $C_6$ on a great circle of $S$.
	\item Set two adjacent vertices of $C_6$, $v$ and $u$, and remove their shared edge.
	\item Place the remaining necessary vertices on $S$ as in \Cref{wheelFig}, and connect them appropriately.
\end{itemize}
We have thus seen that $W_n$ can always be embedded in $\mathbb{R}^3$. Therefore, $\dim W_6 = 2$, and $\dim W_n = 3$ for all $n \neq 6$.
\\

Let us add another wrinkle. Define the \emph{$k^{th}$ degree wheel graph of length n}, denoted $W_n^k$, by $W_n^k = C_n + \epsilon_k$. We have already computed the dimensions of $W_n^1$, but what about $\dim W_n^2$, $\dim W_n^3$, $\dim W_n^k$? When computing the dimension of $W_n$, we relied on the fact that $C_n$ must lie on a unit circle/sphere in $\mathbb{R}^2$/$\mathbb{R}^3$, respectively. But consider $W_4^2$. In the incomplete embedding shown in \Cref{2-wheel}, $C_4$ must be embedded on a circle of radius less than 1, not a unit circle.
\begin{figure}[h!]
\begin{center}
\includegraphics[width=3cm]{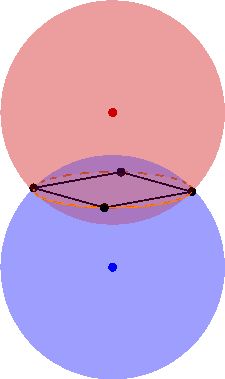}
\end{center}
\caption{The 4-cycle embedded with its vertices a distance 1 away from both the red and blue vertices. The red and blue spheres are the points a distance 1 away from the red and blue vertices, respectively. The orange circle are the points a distance 1 away from both the red and blue vertex.}
\label{2-wheel}
\end{figure}

This turns out to be a common requirement and is precisely why spherical dimension is useful. Recall that the spherical dimension of $G$, denoted $\sdim G$, is the smallest $n$ such that $G$ can be embedded in $\mathbb{R}^n$ with edges length 1 and vertices mapped to distinct points on an n-dimensional sphere of radius less than 1. (We define the spherical dimension of the null graph to be $- \infty$ and the spherical dimension of a vertex is defined to be 0.) Now that we have the concept of spherical dimension, we can state and prove three propositions (\Cref{spheres}, \Cref{Main1}, \Cref{connector}) that greatly simplify computing the dimension of the sum of two graphs. 
 
\begin{lem} \label{preMain1}
Let $S$ and $S'$ be two n-dimensional spheres with distinct centers and with radii $r$ and $q$, respectively. The intersection of $S$ and $S'$ either is empty, contains a single point, or is an (n-1)-dimensional sphere. If the intersection is an (n-1)-dimensional sphere, say $S''$, then $S''$ has the following properties:

(i) The radius of $S''$ is less than or equal to $r$ and $q$.

(ii) If $r=q$, the radius of $S''$ is strictly less than $r=q$.

(iii) If the center of $S$ lies on the surface of $S'$, the radius of $S''$ is strictly less than $r$.
\end{lem}

\begin{proof}
Let $d$ be the distance between the centers of $S$ and $S'$. Without loss of generality, we may place the center of $S$ at the origin and the center of $S'$ at $(d,0,...,0)$. Therefore, 
\[
S = \left \{ (x_1,...,x_n) \text{ }|\text{ } x_1^2+...x_n^2 = r^2 \right \}
\]
and
\[
S' = \left \{ (x_1,...,x_n) \text{ }|\text{ } (x_1-d)^2+...+x_n^2 = q^2 \right \}.
\]
Subtracting $x_1^2+...x_n^2 = r^2$ from $(x_1-d)^2+...+x_n^2 = q^2$ yields a linear equation in $x_1$ alone. Solving this linear equation yields 
\[
x_1 = A
\]
for some constant $A$. Plugging back into the first equation yields 
\[
x_2^2+x_3^2+...+x_n^2 = R,
\]
where 
\[
R = r^2-\left ( \frac{r^2-q^2+d^2}{2d} \right )^2.
\]
Therefore,
\[
S \cap S' = \{ (x_1,...,x_n)\text{ }|\text{ }x_1=A,\text{ }x_2^2+x_3^2+...+x_n^2 = R \}
\]
If $R<0$, no real $x_2,...,x_n$ satisfy $x_2^2+x_3^2+...+x_n^2 = R$, so the intersection is empty. If $R=0$, $x_1 = A$ and $x_2=x_3=...=x_n = 0$ is the only solution to the previous equation, so the intersection is a point. Finally, if $R>0$, the set $S \cap S'$ defines an (n-1)-dimensional sphere.

Suppose the last case hold -- i.e., $S \cap S'$ is an (n-1)-dimensional sphere, $S''$. We now prove that $S''$ has the three desired properties.

(i) All the points of $S''$ are a distance $r$ away from the center of $S$ and a distance $q$ away from the center of $S'$. Since the center of $S''$ is the closest point equidistant from all the points of $S''$, it must be at least as close as the centers of $S$ and $S'$. In other words, $S''$ must have radius less than or equal to both $r$ and $q$.

(ii) If $r=q$, $R=r^2-\frac{d^2}{4}<r^2=q^2$, so  $0<R<r^2$. Since the radius of $S''$ is $\sqrt{R}$, $S''$ has radius less than $r=q$.

(iii) If the center of $S$ lies on the surface of $S'$, $d=q$. Thus, $R=r^2-\frac{r^4}{4q^2}$, so $0<R<r^2$, and so $S''$ has radius less than $r$.
    
We have now shown the three desired properties of $S''$, and so the proof is complete.
\end{proof}

\begin{prop} \label{spheres}
Let $|G|, |H| > 0$. Set an embedding of $G+H$. If either 
\begin{itemize}
\item $|H|>1$, or
\item $G+H$ lies on an n-dimensional sphere,
\end{itemize}
then $G$ lies on an $\sphere[m]$. 
\end{prop}

\begin{proof}
If $|G|=1$, the proof is trivial, so suppose $|G|>1$. Set an embedding of $G+H$ in $\mathbb{R}^n$, and let $v$ be a vertex of $H$. Then $G$ lies on an n-dimensional unit sphere, $S$, centered at $v$. Without loss of generality, let $v$ be at the origin, so 
\[
S = \{(x_1,...,x_n)\text{ }|\text{ }x_1^2+...+x_n^2 = 1\}.
\]

Now suppose that $|H|>1$. Then there exists another vertex of $H$, $v'$, and so $G$ also lies on a unit sphere, $S'$, centered at $v'$. Thus, $G$ lies on the intersection of $S$ and $S'$, so by \Cref{preMain1}, $G$ lies on either the empty set, a point, or a unique (n-1)-dimensional sphere, $S''$. Since $|G|>1$, $G$ does not lie on the empty set or on a point, so $G$ lies on $S''$. Since $S$ and $S'$ both have radius 1, \Cref{preMain1}(ii) yields that the radius of $S''$ is less than the radius of $S$, which is 1. In other words, $G$ lies on an $\sphere[n-1]$, as desired.

Now suppose that $G+H$ lies on a hyper-sphere. Then $G+H$ lies on an n-dimensional sphere, $S'$. Therefore, $G+H$ lies on the intersection of $S$ and $S'$, and so \Cref{preMain1} yields that $G+H$ lies on either the empty set, a point, or an (n-1)-dimensional sphere, $S''$. Since $|G|>1$, $G$ does not lie on the empty set or on a point, so $G$ lies on $S''$. Since the center of $S$, $v$, is a vertex of $H$, the center of $S$ lies on the surface of $S'$, and so \Cref{preMain1}(iii) yields that the radius of $S''$ is less than the radius of $S$, which is 1. In other words, $G$ lies on an $\sphere[n-1]$, as desired.
\end{proof}

\begin{prop} \label{Main1}
Let $G$ be a graph. Set an embedding of $G$ in $\mathbb{R}^{n+m}$ such that $G$ is embedded on a unique n-dimensional sphere $S$ with radius $r$. Then the points a distance $d>r$ away from all the points of $G$ form the unique m-dimensional sphere $S'$ that
\begin{itemize}
\item shares the same center as $S$,
\item lies in a plan orthogonal to the plane $S$ lies in, and
\item has radius $\sqrt{d^2-r^2}$.
\end{itemize}
(Note: We will mostly use this result with $r<d=1$.)
\end{prop}

\begin{proof}
Set an embedding of $G$ in $\mathbb{R}^{n+m}$ such that $G$ is embedded on a unique n-dimensional sphere of radius $r$, $S$. Without loss of generality, let
\[
S = \left \{ (x_1,...,x_{n+m})\text{ }|\text{ }x_1^2+...x_n^2=r^2, \text{ }x_{n+1}=...=x_{m}=0 \right \}.
\]

Now, let $p=(p_1,...,p_{n+m})$ be a point a distance $d$ away from every vertex of $G$, and let $p_{proj} = (p_1,...,p_n,0,...,0)$. For every vertex $v$ of $G$,
\[
|p_{proj}-v| = \sqrt{d^2-p_{n+1}^2-...-p_{n+m}^2}.
\]
In other words, putting 
\[
q = \sqrt{d^2-p_{n+1}^2-...-p_{n+m}^2},
\]
we have that the vertices of $G$ must lie on an $\sphere[n]$ of radius $q$ centered at $p_{proj}$.

Of course, the points of $G$ are also on $S$. Therefore, the points of $G$ must lie on the intersection of two $\sphere[n]$s. By choice of embedding, $G$ cannot be embedded on an $\sphere[n-1]$ (and clearly cannot be embedded on a set of one point or the empty set), so \Cref{preMain1} yields that $p_{proj}$ must be the center of $S$, which is the origin. Since $p_{proj}$ lies on the origin, $p$ must be of the form $(0,...,0,p_{n+1},...,p_{n+m})$.

Now, let $v=(v_1,...,v_n,0,...,0)$ be a vertex of $G$. The Euclidean distance between $p$ and $v$ is
\[
d=\sqrt{v_1^2+...+v_n^2+p_{n+1}^2+...+p_{n+m}^2}.
\]
Since $v_1^2+...+v_n^2=r^2$, this equation can be rewritten as
\[
p_{n+1}^2+...+p_{n+m}^2 = d^2-r^2,
\]
so the set of points a distance $d$ away from all the vertices of $G$ are
\[
\left \{ (x_1,...,x_n)\text{ }|\text{ }x_1=...=x_n=0, x_{n+1}^2+...x_m^2=d^2-r^2 \right \}.
\]
Since $d>r$, $d^2-r^2>0$, and so the above set defines an $\sphere[m]$. Indeed, this $\sphere[m]$ has the same center as $S$ (the origin) and is orthogonal to $S$, as desired.
\end{proof}

\begin{thm} \label{preconnector} 
Let $|G|,|H|>1$. Then the following two statements hold:

Suppose that $G$ and $H$ can be embedded on a $\sphere[n_1]$ and $\sphere[n_2]$, respectively, such that the squares of the radii sum to 1. Then $G+H$ can be embedded in $\mathbb{R}^{n_1+n_2}$. 

Conversely, if $G+H$ can be embedded in $\mathbb{R}^n$, there exists $n_1$ and $n_2$ such that $n_1 + n_2 = n$ and $G$ and $H$ can be embedded on an $\sphere[n_1]$ and $\sphere[n_2]$, respectively, such that the squares of the radii sum to 1.
\end{thm}

\begin{proof}
Suppose that $G$ and $H$ can be embedded in $\sphere[n_1]$ and $\sphere[n_2]$ as in the statement. Let $r_1$ be the radius of $\sphere[n_1]$ and $r_2$ be the radius of $\sphere[n_2]$. Then embedding $G$ on the $\sphere[n_1]$ defined by
\[
\left \{ (x_1,...,x_{n_1+n_2}) \mid x_1^2+...+x_{n_1}^2 = r_1^2, x_{n_1+1}=...=x_{n_1+n_2}=0 \right \},
\]
$H$ on the $\sphere[n_2]$ defined by
\[
\left \{ (x_1,...,x_{n_1+n_2}) \mid x_1=...=x_{n_1}=0, x_{n_1+1}^2+...+x_{n_1+n_2}^2=r_2^2 \right \},
\]
and adding all possible edges between $G$ and $H$ yields an embedding of $G+H$.

Now suppose that $G+H$ can be embedded in $\mathbb{R}^n$. By \Cref{spheres}, $G$ lies on a $\sphere[]$. Choose the unique lowest-dimensional such $\sphere[]$. Let its dimension be $n_1$ and its radius be $r_1$. Then by \Cref{Main1}, $H$ lies on an $\sphere[n-n_1]$ with radius $\sqrt{1-r_1^2}$. Setting $n_2 = n-n_1$ completes the proof.
\end{proof}

\begin{cor} \label{connector}
Let $|G|, |H| > 1$. If $G$ and $H$ can be embedded on a $\sphere[\sdim G]$ and $\sphere[\sdim H]$ such that the squares of the radii sum to 1, 
\[
\dim(G+H) = \sdim G+\sdim H.
\]
Otherwise, 
\[
\dim(G+H) > \sdim G + \sdim H.
\]
\end{cor}

\begin{proof}
This follows easily of \Cref{preconnector}.
\end{proof}

Now that we have \Cref{connector}, computing $\dim W_n^k$ reduces to computing the spherical dimension of $C_n$.

\begin{lem} \label{Final Wheel Lem}
For $n \neq 6$, $\sdim C_n = 2$, but $\sdim C_6 = 3$.
\end{lem}

\begin{proof}
Given \Cref{radii table}, showing that $\sdim C_n = 2$ is equivalent to showing that there exists an $m$ such that $m \nmid n$ and $\frac{1}{6}n < m < \frac{1}{2}n$. For $n<6$, $m=1$ works. For $n=6$, clearly no $m$ exists, but by methods similar to those in \Cref{wheelFig}, $C_6$ can be embedded on a $\sphere[3]$.

To complete the proof, we must show that such an $m$ exists for $n>6$. Set $n>6$. There exists an $m'$ such that $\frac{1}{6}n < m' < m'+1 < \frac{1}{2}n$. If they both divide $n$, then $m'(m'+1)$ divides $n$ because consecutive integers are relatively prime. Since $\frac{1}{6}n < m'$, we conclude $(m'+1) < 6$. We now complete the proof by checking the cases $m' = 1,...,4$:
\begin{itemize}
    \item If $m'=1$, $m'<1<\frac{1}{6}n$, so we are done.
    \item If $m'=2$, $m'>\frac{1}{6}n$, and $m'(m'+1) \mid n$, then $n = 6$, so we are done.
    \item If $m'=3$, $m'>\frac{1}{6}n$, and $m'(m'+1) \mid n$, then $n = 12$. But $m=5$ works, so we are done.
    \item If $m'=4$, $m'>\frac{1}{6}n$, and $m'(m'+1) \mid n$, then $n=20$. But $m=6$ works, so we are done.
\end{itemize}
Thus, the desired $m$ exists for every $n>6$, and so $\sdim C_n = 2$ for all $n>6$.
\end{proof}

\begin{cor} \label{Final Wheel}
The dimensions of the wheel graphs are as follows:
\begin{figure}[H]
\centering
\begin{tabular}{ c | c | c }
 \textbf{Dimension of $W_n^k$} & $n=6$ & $n \neq 6$ \\ \hline
$k=1$     & 2 & 3 \\
$k=2$     & 4 & 3 \\
$k\geq 3$ & 5 & 4 \\
 \end{tabular}
\end{figure}
\end{cor}

\begin{proof}
We already proved the result for $k=1$, so we consider only the cases for $k \geq 2$. Notice that $\epsilon_2$ can be embedded on a $\sphere[1]$ of any radius and that for $j>2$, $\epsilon_j$ can be embedded on a $\sphere[2]$ of any radius. The result now follows from \Cref{Final Wheel Lem} and \Cref{connector}.
\end{proof}


\section{Finding Graphs Minor Minimal with Respect to Dimension: $K_n$}

Once we have computed the dimension of a graph, $G$, we might begin wondering about the dimensions of its minors. This topic is not as simple as it first appears. For example, unlike subgraphs, minors sometimes have greater dimension than the original graph. Indeed, an example we worked with in the previous section illustrate this point: $W_5$ is a minor of $W_6$, but 
\[
\dim W_5 = 3 > 2 = \dim W_6.
\]

In this section, we restrict our focus to whether a graph $G$ is minor minimal with respect to dimension. Recall that $G$ is minor minimal with respect to dimension if every proper minor of $G$ has dimension less than $G$.

For example, consider $K_3$. Since $K_3$ has dimension 2 and all of its minors have dimension at most 1, $K_3$ is minor minimal with respect to dimension 2. 

It turns out that $K_n$ ($n \geq 3$) is minor minimal with respect to dimension $n-1$. To understand why, we must first explore some other special properties of $K_n$. In particular, notice that $K_3$'s embedding is ``unique'' and ``compact.'' More precisely, every embedding of $K_3$ is a triangle, and the vertices of $K_3$ lie on a circle of radius less than $\frac{\sqrt{2}}{2}$. The following proposition shows that these two properties are very useful:

\begin{prop} \label{Main2}
Set an embedding of $G$ such that $G$ lies on a unique $\sphere[n]$, which has radius $r$. Place a vertex a distance 1 away from all the vertices of $G$ (which is possible by \Cref{Main1}) and adjoin it to $G$. The result is an embedding of $G + vertex$ which lies uniquely on an (n+1)-dimensional sphere of radius $R=\frac{1}{2\sqrt{1-r^2}}$.
\end{prop}

\begin{proof}
Set an embedding of $G$ such that $G$ is embedded on a unique $\sphere[n]$, $S$, which has radius $r$. Without loss of generality, let
\[
S = \left \{ (x_1,...,x_{n+m}\suchthat x_1^2+...+x_n^2=r^2 \text{ and } x_{n+1} = ... = x_{n+m} = 0 \right \}. 
\]
From \Cref{Main1}, we know that one of the points a distance 1 away from $S$ is the point defined by $x_{n+1} = \sqrt{1-r^2}$ and all other $x_i = 0$. Call this point $p_v$. Without loss of generality, we construct $G+vertex$ by placing a vertex, $v$, on $p_v$ and adjoining $v$ to $G$.

We must now show that $G+vertex$ lies on a unique (n+1)-dimensional sphere, $S'$, which has radius $R=\frac{1}{2\sqrt{1-r^2}}$. Consider a potential center of $S'$, $p$. \Cref{Main1} guarantees that the points equidistant from $G$ must lie in the last $m$ coordinates, so $p$ is of the form $(0,...,0,x_{n+1},...,x_{n+m})$. Moreover, $p$ must be of the form $(x_1,...,x_{n+1},0,...,0)$ if $S'$ is to both be $n+1$ dimensional and contain $p_v$ and $S$. Therefore, $p = (0,...,0,p_{n+1},0,...,0)$ for some real number $p_{n+1}$.

Letting $p_S$ be any point on $S$, we have
\[
|p-p_S| = \sqrt{r^2+p_{n+1}^2}
\]
and
\[
|p-p_v| = \sqrt{1-r^2} - p_{n+1}.
\]
For $p$ to be the center of $S$, $p$ must be equidistant from both $p_v$ and $p_S$, so the two expressions must be equal. Solving this equality yields
\[
p_{n+1} = \frac{1-2r^2}{2\sqrt{1-r^2}}.
\]
Plugging back into $\left | p-p_v \right |$ yields the radius of $S'$, $\frac{1}{2\sqrt{1-r^2}}$, as desired. Since $S$ has uniquely defined both the center and radius of $S'$, it has uniquely defined $S'$, and so we're done.
\end{proof}

\begin{cor} \label{Main2Cor}
Set an embedding such that $G$ lies on a unique $\sphere[n]$ of radius $r$. Adjoin a vertex, $v$, and let $R$ be the radius of the unique (n+1)-dimensional sphere on which $G+v$ lies (as guaranteed by \Cref{Main2}).  

\textbf{a.} If $r < \frac{\sqrt{2}}{2}$, $r< R < \frac{\sqrt{2}}{2}$.

\textbf{b.} If $r = \frac{\sqrt{2}}{2}$, so does $R$.

\textbf{c.} If $r > \frac{\sqrt{2}}{2}$, $R-r>0$ and increases as $r$ increases.
\end{cor}

\begin{proof}
\Cref{add vertex 1} shows that $r<R$ if $r \neq \frac{\sqrt{2}}{2}$ and $r=R$ if $r=\frac{\sqrt{2}}{2}$. It should also be obvious from \Cref{add vertex 1} that $R-r$ increases as $r$ increases if $r>\frac{\sqrt{2}}{2}$ ($R$ will be bigger and the n-dimensional sphere will be farther above the center of (n+1)-dimensional sphere). \Cref{add vertex 2} shows that $R<\frac{\sqrt{2}}{2}$ if $r<\frac{\sqrt{2}}{2}$.
\end{proof}

\begin{figure}[h]
\centering
\includegraphics[width=5cm]{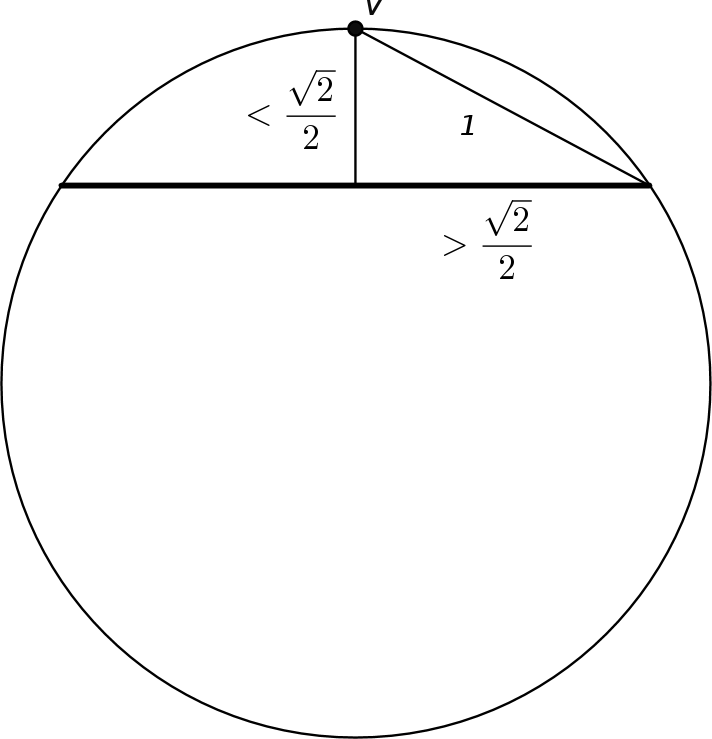}
\hspace{0.5cm}
\includegraphics[width=5cm]{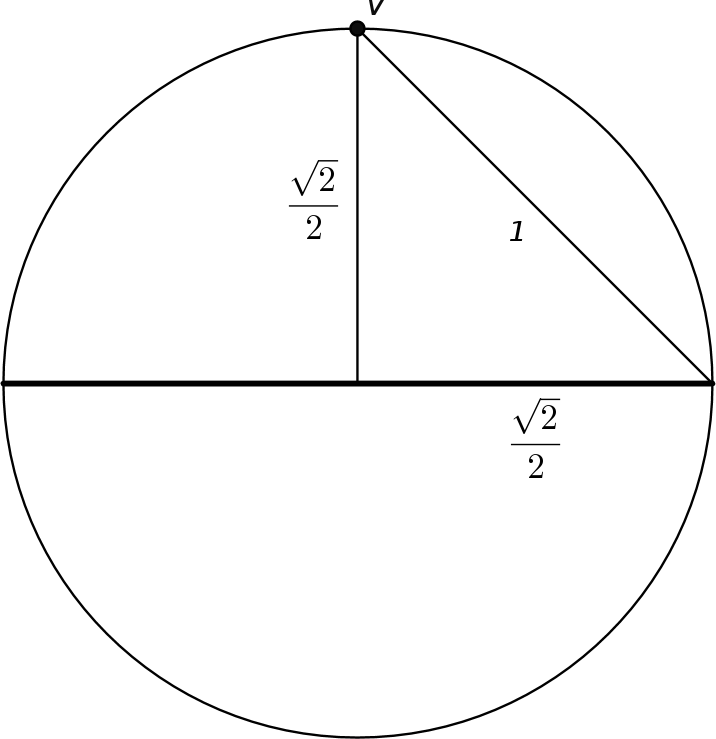}
\hspace{0.5cm}
\includegraphics[width=5cm]{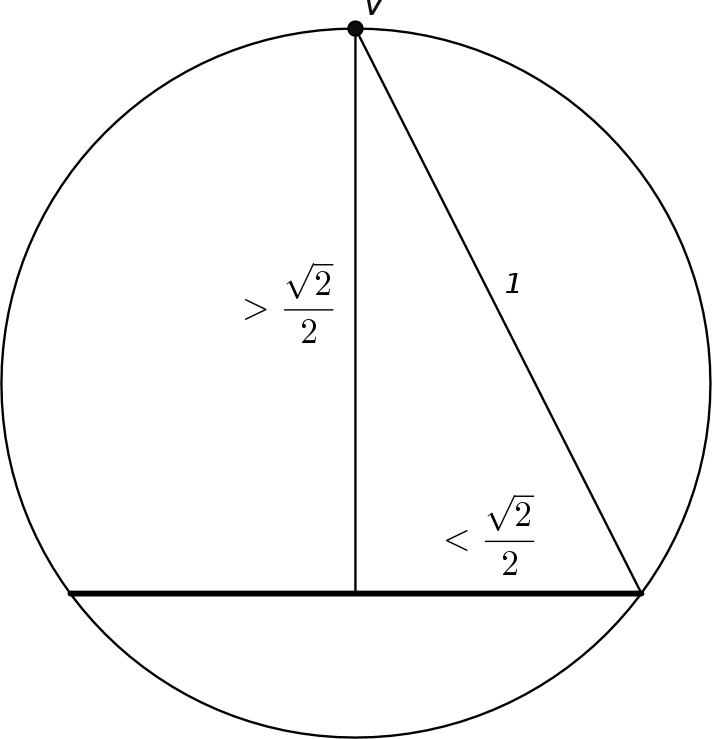}
\caption{The (n+1) and n-dimensional sphere projected into two dimensions (the circle and the thick line, respectively). If the n-dimensional sphere had radius greater than, equal to, or less than $\frac{\sqrt{2}}{2}$, it lies above, at, below the center of the (n+1)-dimensional sphere, respectively. }
\label{add vertex 1}
\end{figure}
\begin{figure}[h]
\begin{center}
\includegraphics[width=5cm]{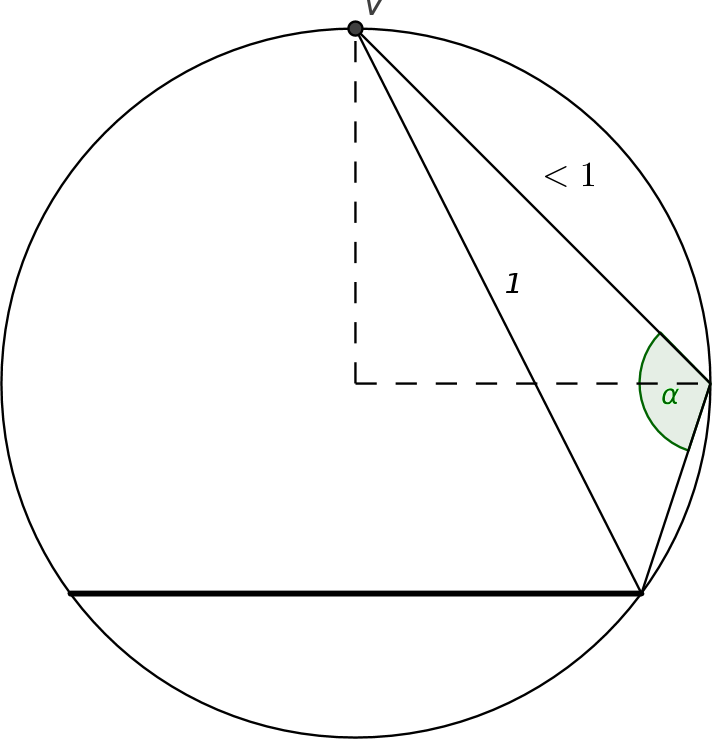}
\end{center}
\caption{If the n-dimensional sphere lies below the center of the (n+1)-dimensional sphere, $\alpha > \frac{\pi}{2}$. Therefore, the right triangle constructed from $v$ and two radii has hypotenuse less than 1, so the radius of the (n+1)-dimensional sphere is less than $\frac{\sqrt{2}}{2}$.}
\label{add vertex 2}
\end{figure}

Since $K_3$ has a unique embedding, \Cref{Main2} yields a unique embedding of $K_4$. Since $K_3$'s unique embedding is on a circle of radius less than $\frac{\sqrt{2}}{2}$, \Cref{Main2Cor} shows that $K_4$'s unique embedding is as well. And so on.

\begin{prop} \label{KSphere}
The graph $K_n$ has a unique embedding. In this embedding, $K_n$ lies on a unique $\sphere[n-1]$, which has radius less than $\frac{\sqrt{2}}{2}$.
\end{prop}
\begin{proof}
We proceed by induction. We have already seen the case when $n=3$. For some $n\geq 3$, suppose $K_n$ has a unique embedding in which $K_n$ lies on an $\sphere[n-1]$ of radius $r < \frac{\sqrt{2}}{2}$ in this embedding. Set an embedding of $K_{n+1}$. Remove a vertex, $v$, from $K_{n+1}$ to get $K_n$. $K_n$ lies on an $\sphere[n-1]$ of radius $r$. Adding back $v$, \Cref{Main2Cor} yields that $K_{n+1}$ lies on a unique $\sphere[n]$, which has radius less than $\frac{\sqrt{2}}{2}$, as desired.
\end{proof}

\begin{cor} \label{KSDim}
The graph $K_n$ has dimension and spherical dimension $n-1$.
\end{cor}
\begin{proof}
This is a direct consequence of \Cref{KSphere}
\end{proof}

We have now completed the first step of showing that $K_n$ is minor minimal, computing the dimension of $K_n$. We must now complete the second step, showing that every proper minor of $K_n$ has lesser dimension:

\begin{lem} \label{KMin}
Let $H$ be a minor of $K_n$ (denoted $H \prec K_n$). Then $\dim H < n-1$.
\end{lem}
\begin{proof}
We proceed by induction. We have already completed the base case. Assume every minor of $K_n$ has dimension less than $n-1$. We must show that $\dim H < n$ for every $H \prec K_{n+1}$. To do this, we consider two exhaustive cases:
\begin{description}[style=unboxed]
	\item[$H$ is obtained by removing at least one vertex or contracting at least one edge] Then $|H| \leq n$, so $H$ is a subgraph of $K_n$, and so $\dim H \leq \dim K_n = n-1 < n$, as desired.
	\item[$H$ is obtained by only removing edges] It suffices to show that $\dim K_{n+1} - edge < n$, so consider $K_{n+1} - edge$. Label the two vertices incident to the removed edge $v_1$ and $v_2$, and remove them both. The resulting subgraph of $K_{n+1}-edge$ is $K_{n-1}$, so \Cref{KSphere} yields that it lies on an $\sphere[n-2]$. Therefore, by \Cref{Main1}, there are two points a distance 1 from all the vertices of $K_{n-1}$ in $\mathbb{R}^{n-1}$. Placing $v_1$ and $v_2$ on these points and adjoining them to $K_{n-1}$ results in an embedding of $K_{n+1}-edge$ in $\mathbb{R}^{n-1}$. Thus, $\dim K_{n+1}-edge \leq n-1 < n$, as desired.
\end{description}
Therefore, regardless of how the minor $H \prec K_{n+1}$ is derived, $\dim H < n$. This completes the proof by induction.
\end{proof}

All together:

\begin{thm} \label{MainK}
The graph $K_n$ is minor minimal with respect to dimension $n-1$.
\end{thm}
\begin{proof}
This is a direct consequence of \Cref{KSphere} and \Cref{KMin}.
\end{proof}

Our work with $K_n$ now allows us to prove a few important (albeit, peripheral) facts: 

\begin{thm} \label{exist}
Every graph has both a dimension and a spherical dimension.
\end{thm}
\begin{proof}
Let $G$ be a graph with at least two vertices. $G$ is a subgraph of $K_n$ for some $n$. Therefore, $G$ can be embedded in $\mathbb{R}^{n-1}$. There exists some $1 \leq m \leq n-1$ such that $G$ can be embedded in $\mathbb{R}^{m}$ but not $\mathbb{R}^{m-1}$; $\dim G = m$.

That $G$ has a spherical dimension follows identically.
\end{proof}

\begin{prop} \label{addV}
For any graph, $G$, $\sdim(G+vertex) > \sdim G$. 
\end{prop}

\begin{proof}
This is a direct consequence of \Cref{spheres} and \Cref{Main2}.
\end{proof}

It is tempting to also claim that $\dim(G+vertex)> \dim G$ or $\dim(G+vertex) > \sdim G$, but this is not generally true. For example, we have already seen that $\dim C_6 = 2$ and $\sdim C_6 = 3$, but $C_6 + vertex = W_6$ has dimension 2.

Another interesting question is whether there exists a graph, $G$, such that $\sdim(G + vertex) > \sdim G + 1$. This is an open question. However, notice that it would suffice to find a $G$ that cannot be embedded on an $\sphere[\sdim G]$ of radius less than or equal to $\frac{\sqrt{2}}{2}$ (to see this, iteratively apply \Cref{Main2} to $G$). This question is briefly revisited at the beginning of Section 7.


\section{Finding Graphs Minor Minimal with Respect to Spherical Dimension: $S_n$}

We have found a class of graphs that is minor minimal with respect to dimension. Now let us find a class of graphs that is minor minimal with respect to spherical dimension. In particular, we will find the graphs with the fewest number of vertices that are minor minimal with respect to spherical dimension $n-1$.

 Since $\sdim K_{n-1} = n-2$ and $\sdim K_{n} = n-1$, we know such a graph must have $n$ vertices. Let us give these graphs a name: Let $\mathbb{S}_n$ denote the set of graphs that are minor minimal with respect to spherical dimension $n-1$ and have $n$ vertices, and let $S_n$ be some graph $S_n \in \mathbb{S}_n$. Computing $\mathbb{S}_n$ is relatively easy for small $n$:

\begin{prop} \label{sBase}
The following are $\mathbb{S}_n$ for small $n$:

\textbf{a.} $\mathbb{S}_1 = \{ \epsilon_1 \}$.

\textbf{b.} $\mathbb{S}_2 = \{ \epsilon_2 \}$.

\textbf{c.} $\mathbb{S}_3 = \{ \epsilon_3 \}$.

\textbf{d.} $\mathbb{S}_4 = \{ K_{3,1} \}$ (equivalently, $\{ K_1 + \epsilon_3\}$).

Indeed, $\epsilon_1$, $\epsilon_2$, and $\epsilon_3$ are the only minor minimal graphs with respect to their respective spherical dimensions.
\end{prop}

\begin{proof}
\textbf{a.} Recall that $\sdim \epsilon_0 = - \infty$ and $\sdim \epsilon_1 = 0$. Since the only minor of $\epsilon_1$ is the empty graph, $\epsilon_1$ is minor minimal with respect to spherical dimension 0. Moreover, since every graph that is not the empty graph contains $\epsilon_1$ as a subgraph, $\epsilon_1$ is the only minor minimal graph with spherical dimension 0.
\\
\\
\textbf{b.} Since $\epsilon_2$ can be embedded on two points but not a single point, $\sdim \epsilon_2 = 1$. Since all of $\epsilon_2$'s minors are subgraphs of $\epsilon_1$, which we just saw has spherical dimension 0, $\epsilon_2$ is minor minimal with respect to spherical dimension 1. Moreover, any graph that is not the empty graph or $\epsilon_1$ contains $\epsilon_2$ as a subgraph, so $\epsilon_2$ is the only minor minimal graph with respect to spherical dimension 1.
\\
\\
\textbf{c.} Since $\epsilon_3$ can be embedded on a circle but not two points, $\sdim \epsilon_3 = 2$. Since all of $\epsilon_3$'s minors are subgraphs of $\epsilon_2$, which we just saw has spherical dimension 1, $\epsilon_3$ is minor minimal with respect to spherical dimension 2. All graph with two or fewer vertices, have spherical dimension at most 1, so since all graph with three or more vertices contain $\epsilon_3$ as a subgraph, $\epsilon_3$ is the only minor minimal graph with respect to spherical dimension 2. 
\\
\\
\textbf{d.} Since $\sdim \epsilon_3 = 2$ and can be embedded on circles of arbitrarily small radius, \Cref{Main2} yields that $\sdim K_1 + \epsilon_3 = 3$. We now show that $\sdim H < 3$ for all $H \prec K_{3,1}$ by considering two exhaustive cases:
\begin{description}[style=unboxed]
    \item[$H \prec K_{3,1}$ is obtained by removing at least one vertex or contracting at least one edge] Then $H \subgraph K_3$, and so $\sdim H \leq 2$. 
    \item[$H$ is obtained by only removing edges] Then $H \subgraph P_3 \cupdot vertex$ (where $P_3$ is the path on 3 vertices). Since $\sdim P_3 \cupdot vertex =2$, $\sdim H \leq 2$.
\end{description}
Since every minor of $K_{3,1}$ falls into one of the two above categories, every minor of $K_1 + \epsilon_3$ has spherical dimension at most 2, and so $K_{3,1}$ is minor minimal with respect to spherical dimension 3. 

Finally, we show that $K_{3,1}$ is the only member of $\mathbb{S}_4$. To do this, we must show that every subgraph of $K_4$ that is not a supergraph of $K_{3,1}$ has spherical dimension at most 2. The following more general lemma suffices:
\end{proof}

\begin{lem} \label{4vSubgraphs}
Let $K_{3,1} \subgraph G \subgraph K_4$. Every proper minor of $G$ that is not a supergraph of $K_{3,1}$ can be embedded on a circle of radius less than or equal to $\frac{\sqrt{2}}{2}$. 
\end{lem}

\begin{proof}
We need to consider the minors of the four graphs listed in \Cref{4v3d}. 
\begin{figure}[!ht]
\centering
\includegraphics[width=3cm]{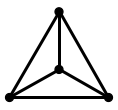} 
\includegraphics[width=3cm]{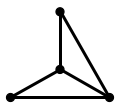}
\includegraphics[width=3cm]{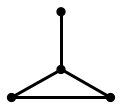}
\includegraphics[width=3cm]{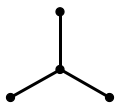}
\caption{The four graphs that are both subgraphs of $K_4$ and supergraphs of $K_{3,1}$.}
\label{4v3d}
\end{figure} 
Since $K_3$ can be embedded on a circle of radius less than $\frac{\sqrt{2}}{2}$, it suffices to consider only the minors obtained by removing one edge from any of these four graphs.
\begin{description}
    \item[Minors of $K_4$] Up to symmetry, every subgraph of $K_4$ obtained by removing a single edge is $K_4 - edge$, which is a supergraph of $K_{3,1}$ and so considered below. 
    \item[Minors of $K_4 - edge$] Up to symmetry, one can obtain two subgraphs of $K_4 - edge$ by removing a single edge: $C_4$ and $K_{3,1}+edge$. 
    \begin{description}
        \item[$C_4$] Embed $C_4$ as a square. All the vertices of $C_4$ are a distance $\frac{\sqrt{2}}{2}$ away from the center of the square, so $C_4$ can be embedded on a $\sphere[2]$ of radius $\frac{\sqrt{2}}{2}$.
	\item[$K_{3,1} + edge$] $K_{3,1}+edge$ is a supergraph of $K_{3,1}$ and so is considered below.
    \end{description}
    \item[Minors of $K_{3,1} + edge$] Up to symmetery, one can obtain three different subgraphs of $K_{3,1} + edge$ by removing a single edge: $P_4$, $K_3 \cupdot vertex$, and $K_{3,1}$.
    \begin{description}
        \item[$P_4$] $P_4 \subgraph C_4$, which we just saw can be embedded on a $\sphere[2]$ of radius $\frac{\sqrt{2}}{2}$.
        \item[$K_3 \cupdot vertex$] By \Cref{KSphere}, $K_3$ can be embedded on a $\sphere[2]$ of radius less than $\frac{\sqrt{2}}{2}$. Embed $K_3$ on such a $\sphere[2]$ and add another vertex anywhere on the $\sphere[2]$ that does not already have a vertex to get an embedding of $K_3 \cupdot vertex$ on a $\sphere[2]$ of radius less than $\frac{\sqrt{2}}{2}$.
	\item[$K_{3,1}$] $K_{3,1}$ is, of course, a subgraph of itself, and so we consider it below.
    \end{description}
   \item[Minors of $K_{3,1}$] Up to symmetry, the only subgraph is $P_3 \cupdot vertex$, which is a subgraph of many of the graphs already considered and so can be embedded on a $\sphere[2]$ of radius less than $\frac{\sqrt{2}}{2}$. 
\end{description}
We have considered all possible minors obtained by removing one edge from one of the desired graphs, and so we are done.
\end{proof}

To prove the general claim $\mathbb{S}_n = \{ K_{n-3} + \epsilon_3\}$ for $n \geq 4$, we need a few more lemmas. For the next lemma, we need the following definition. Let 
\[
R_G = \inf_{S \in \left\{\substack{\text{$\sphere[\sdim G]$s on which } \\ \text{ $G$ can be embedded}}\right\}} \text{radius of $S$}
\]

\begin{lem} \label{4verts}
Let $G$ be a graph such that $\sdim G = 2$ and $|G|= 4$. Then $R_G \leq \frac{\sqrt{2}}{2}$
\end{lem}

\begin{proof}
If $G$ is a supergraph of $K_{3,1}$, it has spherical dimension greater than 2 (\Cref{sBase}d). Otherwise, $G$ can be embedded on a circle of radius less than or equal to $\frac{\sqrt{2}}{2}$ (\Cref{4vSubgraphs}). This completes the proof.
\end{proof}

\begin{lem} \label{K+E3 sdim}
For $n \geq 4$, $\sdim (K_{n-3} + \epsilon_3) = n-1$.
\end{lem}

\begin{proof}
Since $K_{n-3} + \epsilon_3 \subgraph K_n$, $\sdim (K_{n-3} + \epsilon_3) \leq n-1$. To get the reverse inequality, we proceed by induction. \Cref{sBase}d provides the base case. Suppose that $\sdim (K_{n-3}+\epsilon_3) \geq n-1$. Since $K_{n-2}+\epsilon_3 = (K_{n-3} + \epsilon_3) + vertex$, \Cref{addV} yields that $\sdim (K_{n-2}+\epsilon_3)\geq n$, completing the proof.
\end{proof}

The next lemma and proof uses the following two definitions:
\begin{itemize}
	\item A vertex of $G$ is called a \emph{center vertex} if it is adjacent to every other vertex of $G$.
	\item $G +_? H$ denotes the set of all graphs that can be constructed from $G \cupdot H$ by adding some edges (possibly none or all) between $G$ and $H$.
\end{itemize}

\begin{lem} \label{comb}
Let $G$ be a graph with $n \geq 4$ vertices. Suppose that $H$ is a subgraph of $K_{3,1}$ for every subgraph, $H$, of $G$ obtained by only removing vertices and with $|H| = 4$ Then $K_{n-3} + \epsilon_3 \subseteq G$. 
\end{lem}

\begin{proof}
If $G = K_n$, we are done, so suppose $G \neq K_n$. Then there exists non-adjacent vertices $v_1$ and $v_2$. Let $H$ the graph obtained by removing all but the following four vertices from $G$: $v_1$, $v_2$, and some arbitrary $u_1$ and $u_2$.  Since $K_{3,1}$ is a subgraph of $H$, $H$ has a center vertex. Since $v_1$ and $v_2$ are not adjacent to one another, neither can be center vertices of $H$, so either $u_1$ or $u_2$ is a center vertex of $H$. Regardless, $u_1$ and $u_2$ are adjacent. 

Since $u_1$ and $u_2$ were arbitrary vertices of $G$, not $v_1$ or $v_2$, any pair of vertices of $G$ that do not include $v_1$ or $v_2$ are adjacent. In other words, $G \in K_{n-2} +_? \epsilon_2$, where $\epsilon_2$'s two vertices are $v_1$ and $v_2$. To determine which edges exist between $K_{n-2}$ and $\epsilon_2$, we consider two exhaustive cases:
\begin{itemize}
    \item \textbf{Suppose every vertex of $K_{n-2}$ is adjacent to both $v_1$ and $v_2$.} Then $G = K_{n-2} + \epsilon_2$, which is a supergraph of $K_{n-3} + \epsilon_3$, and we are done.
    \item \textbf{Suppose there exists at least one vertex, $w$, of $K_{n-2}$ that is not adjacent to at least one of $v_1$ or $v_2$.} Let $H$ be the graph obtained by removing all but the following four vertices from $G$: $v_1$, $v_2$, $w$, and some arbitrary $u$. Since $v_1$ and $v_2$ are not adjacent, and $w$ is not adjacent to at least one of $v_1$ and $v_2$, $u$ must be the center vertex of $H$. In other words, $u$ is adjacent to $v_1$, $v_2$, and $w$. Since $u$ was arbitrary, every vertex of $G$ not $v_1$, $v_2$, or $w$ is adjacent to each of $v_1$, $v_2$, and $w$. Therefore, $G$ is a supergraph of $K_{n-3} + \epsilon_3$, where the $\epsilon_3$'s vertices are $v_1, v_2, w$.
\end{itemize}
Either way, $K_{n-3} + \epsilon_3$ is a subgraph of $G$, as desired.
\end{proof}

\begin{prop} \label{jump}
Let $G$ be a graph. There exists some $n \in \mathbb{N}$ such that 
\[
\sdim (G+K_n) > \sdim G + n
\]
if and only if $R_G > \frac{\sqrt{2}}{2}$. 
\end{prop}

The proof of \Cref{jump} requires many iterative applications of \Cref{Main2}. To simplify the proof, we define $R(r) = \frac{1}{2\sqrt{1-r^2}}$ and $R^{n}$ ($n \in \mathbb{Z}$) as the composition of $R(x)$ $n$ times. $R(x)$ has the following properties:

\begin{lem} \label{RProps}
Let $0<r,r_1,r_2<1$.

\textbf{a.} For all $n \in \mathbb{Z}^+$, $r \leq R^n(r)$.

\textbf{b.} For all $n \in \mathbb{Z}^+$, if $r_1 < r_2$ and $R^{n-1}(r_2)<1$, $R^n(r_1) < R^n(r_2)$.

\textbf{c.} For any $n \in \mathbb{Z}^+$, there exists an $r$ such that $\frac{\sqrt{2}}{2} < r < 1$ and $R^n(r) < 1$.

\textbf{d.} For any $r > \frac{\sqrt{2}}{2}$, there exists an $n$ such that $R^n(r) > 1$. 

\textbf{e.} Let $G$ be a graph, and fix an embedding such that $G$ lies on a unique $\sphere[m]$, which has radius $r$. 
\begin{enumerate}
    \item If $r \leq \frac{\sqrt{2}}{2}$, then for all $i \in \mathbb{Z}^+$, $G+K_i$ lies on a unique $\sphere[m+i]$, which has radius $R^i(r)$.
    \item Let $r \geq \frac{\sqrt{2}}{2}$, and let $n$ be the smallest positive integer such that $R^n(r) > 1$. For all $i \leq n$,  $G+K_i$ lies on a unique (n+i)-dimensional sphere, which has radius $R^i(r)$. For $i>m$, $G+K_i$ cannot be embedded on a q-dimensional sphere for any $q$.
\end{enumerate}
\end{lem}
\begin{proof}
\textbf{a.} We proceed inductively on $n$. The equation $r=R^0(r)$ takes care of the base case. Suppose that $r<R^n(r)$. By \Cref{Main2Cor}, 
\[
r < R^n(r) < R(R^n(r)) = R^{n+1}(r),
\]
completing the inductive step and thus the proof.
\\
\\
\textbf{b.} We proceed inductively on $n$. By choice of $r_1$ and $r_2$, $R^0(r_1) = r_1 < r_2 = R^0(r_2$). Suppose $R^n(r_1) < R^n(r_2) < 1$. Then
\[
R^{n+1}(r_2) = \frac{1}{2\sqrt{1-(R^n(r_2))^2}} > \frac{1}{2\sqrt{1-(R^n(r_1))^2}} = R^{n+1}(r_1),
\]
completing the inductive step and thus the proof.
\\
\\
\textbf{c.} We proceed by induction. The base case is trivial. Suppose there exists an $r$ such that $\frac{\sqrt{2}}{2} < r < 1$ and $R^n(r) < 1$. We must find an $r'$ such that $R^{n+1}(r') < 1$. Consider $r' = R^{-1}(r)$. $R^{n+1}(r') = R^{n}(r) < 1$, and by \Cref{Main2Cor}, $\frac{\sqrt{2}}{2}<r'<r$. This completes the inductive step and thus the proof.   
\\
\\
\textbf{d.} Put $\delta = R(r) -r$ and choose $n = \left \lceil \frac{1 - r}{\delta} \right \rceil$. By \Cref{Main2Cor}c, $R^{i}(r) - R^{i-1}(r) \geq \delta$ for all $i \geq 1$. Therefore, 
\[
R^n(r) = r + \sum_{i=1}^{n} \left ( R^{i}(r) - R^{i-1}(r) \right ) \geq n \delta + r \geq (1-r) + r = 1, 
\]
as desired.
\\
\\
\textbf{e.} Set an embedding such that $G$ lies on an n-dimensional sphere of radius $r$. 
\\
\\
(1) First let $r \leq \frac{\sqrt{2}}{2}$. We proceed inductively. By definition of $r$, $G$ lies on a unique $\sphere[m]$ of radius $R^0(r)=r$. For $0 \leq i$, suppose that $G+K_i$ lies on a unique $\sphere[m+i]$, which has radius $R^i(r)$. We must show that $G+K_{i+1}$ lies on a unique $\sphere[m+i+1]$, which has radius $R^{i+1}(r)$. Well, removing a vertex yields $G+K_i$, which by the inductive hypothesis, lies on a unique $\sphere[m+i]$ of radius $R^i(r)$. By \Cref{Main2}, re-adjoining the removed vertex yields a graph that lies on a unique $\sphere[m+i+1]$ of radius $R(R^i(r))=R^{i+1}(r)$. This completes the inductive step and thus the proof of the first statement.
\\
\\
(2) For $i \leq n$, the proof is identical to the proof of (1). Suppose $i>n$. Since $G+K_n$ does not lie on a $\sphere[]$, the converse of $\Cref{spheres}$ yields that $G+K_i$ does lie on a q-dimensional sphere for any $q$.
\end{proof}

We now prove \Cref{jump}. 

\begin{pf}{\Cref{jump}}
First suppose that $R_G \leq \frac{\sqrt{2}}{2}$, and let $n$ be given. By \Cref{RProps}c, there exists some $r_m > \frac{\sqrt{2}}{2}$ such that $R^n(r_m) < 1$. Since $R_G \leq \frac{\sqrt{2}}{2}$, we may choose an embedding of $G$ such that $G$ lies on a $\sphere[\sdim G]$ of radius $r < r_m$. Fix such an embedding. By \Cref{RProps}b, $R^n(r) < R^n(r_m) = 1$, so by \Cref{RProps}e(1), $G+K_n$ is embedded on a $\sphere[\sdim G + n]$. Thus, $\sdim (G+K_n) \leq \sdim G + n$, completing the first half of the proof.

Now suppose that $R_G > \frac{\sqrt{2}}{2}$. Choose $n$ such that $R^n(R_G) >1$ (\Cref{RProps}d guarantees such a choice of $n$). If $G+K_n$ can be embedded on a $\sphere[\sdim G +n]$, it can be constructed from an embedding of $G$, so set an embedding of $G$. If $G$ lies on a $\sphere[\sdim G]$, \Cref{RProps}e(2) shows that we cannot construct $G+K_n$ on a $\sphere[\sdim G + n]$. If $G$ lies on a $\sphere[]$ of greater dimension, iterative application of \Cref{addV} shows that we cannot construct $G+K_n$ on a $\sphere[\sdim G + n]$. Finally, if $G$ does not lie on any $\sphere[]$, \Cref{spheres} shows that we cannot extend the embedding of $G$ to an embedding of $G+K_1$ and so certainly not to an embedding of $G + K_n$. Regardless, $G+K_n$ cannot be embedded on a $\sphere[\sdim G + n]$, so $\sdim (G + K_n) > \sdim G + n$, as desired.
\end{pf}

We have now developed the necessary lemmas to compute $\mathbb{S}_n$:

\begin{thm} \label{Sn}
For $n > 3$, $\mathbb{S}_n = \{ K_{n-3} + \epsilon_3 \}$.
\end{thm}
\begin{proof}
Let $S_n \in \mathbb{S}_n$. Remove all but four vertices from $S_n$ to get $H$. If for every $H$ obtained this way, $K_{3,1}$ is a subgraph of $H$, then \Cref{comb} yields that $K_{n-3} + \epsilon_3$ is a subgraph of $S_n$. Indeed, since $S_n$ is minor minimal with respect to spherical dimension $n-1$ and $\sdim (K_{n-3} + \epsilon_3) = n-1$ (\Cref{K+E3 sdim}), $S_n$ would equal $K_{n-3} + \epsilon_3$. Thus, it suffices to show that every subgraph of $S_n$ with four vertices obtained by only removing vertices contains $K_{3,1}$. 

Since $\sdim K_3 = 2$ and $\mathbb{S}_4 = \{ K_{3,1} \}$, $K_{3,1}$ is the only graph with four or fewer vertices that is minor minimal with respect to spherical dimension 3, so it suffices to show that $\sdim H = 3$. Suppose to the contrary that $\sdim H \neq 3$. Since $\sdim K_4 = 3$, $\sdim H < 3$. By \Cref{4verts}, there are no graphs with four vertices that have spherical dimension 2 but cannot be embedded on a circle of radius less than or equal to $\frac{\sqrt{2}}{2}$. Therefore, by \Cref{jump}, $\sdim (H+K_{n-4} ) = n-2$. But, by construction of $H$, $S_n$ is a subgraph of $H+K_{n-4}$. Therefore, $\sdim S_n \leq \sdim H+K_{n-4} = n-2$. This contradicts that $\sdim S_n = n-1$. Thus, $\sdim H = 3$, so $K_{3,1} \subseteq H$, and so $S_n = K_{n-3}+\epsilon_3$.
\end{proof}

Before moving on, we use our new machinery to improve upon our earlier result, \Cref{KSphere}:

\begin{prop} \label{KSphereImp}
For $n>1$, $K_n$ has a unique embedding, which is on an $\sphere[n-1]$ of radius $R^{n-2}\left (\frac{1}{2} \right )$ (which is less than $\frac{\sqrt{2}}{2}$).
\end{prop}

\begin{proof}
We proceed by induction. $K_2$ has a unique embedding, which is on a $\sphere[1]$ of radius $R^0\left ( \frac{1}{2} \right ) = \frac{1}{2}$. The inductive step follows directly from \Cref{Main2}.
\end{proof}


\section{Another Class of Graphs Minor Minimal with Respect to Dimension: $S_n + \epsilon_3$}

The graphs $S_n + \epsilon_3$ are our second class of minor minimal graphs with respect to dimension. Before showing this, however, we must first prove a lemma.

\begin{lem} \label{radii}
For $n<4$, $S_n$ can be embedded on an $\sphere[n-1]$ of any radius $0<r<1$. For $n \geq 4$, $S_n$ can be embedded on an $\sphere[n-1]$ with any radius $r$ such that $R^{n-4} \left (\frac{1}{2} \right ) < r < 1$, but cannot be embedded on any $\sphere[n-1]$ of radius $r < R^{n-4} \left (\frac{1}{2} \right )$.
\end{lem}

\begin{proof}
For $n<4$, the proof is trivial. For $n \geq 4$, $S_n$ contains a copy of $K_{n-2}$ as a subgraph (take all $n-3$ vertices of $K_{n-3}$ and any one vertex of $\epsilon_3$), so by \Cref{KSphereImp}, $S_n$ cannot have radius less than $R^{n-4} \left ( \frac{1}{2} \right )$.

To complete the proof, we show that $S_n$ can be embedded on an $\sphere[n-1]$ with any radius between $R^{n-4} \left (\frac{1}{2} \right ) < r < 1$. We proceed inductively. Embed $S_4$ as follows. Place the three non-center vertices on a $\sphere[2]$, $S$, place the center vertex as per \Cref{Main2}, and add the appropriate edges. $S_4$ is now embedded on a sphere $S'$. If we send the radius of $S$ to 0, \Cref{Main2} shows that the radius of $S'$ goes to $\frac{1}{2}$. If we send the radius of $S$ to 1, the radius of $S'$ goes to infinity. Since the radius of $S'$ increases continuously as the radius of $S$ increases, $S'$ can be a $\sphere[3]$ of any radius $R^0 \left (\frac{1}{2} \right ) = \frac{1}{2} < r < 1$.

Now, suppose that $S_n$ can be embedded on an $\sphere[n-1]$ of any radius $R^{n-4}\left (\frac{1}{2} \right) < r < 1$ and set any such an embedding. Since $\sdim S_n = n-1$, $S_n$ lies on a unique $\sphere[n-1]$. Thus, \Cref{Main2} allows us to construct $S_{n+1}$ so that it lies on an $\sphere[n]$ of radius $R(r)$. Letting $r$ run from $R^{n-4}$ to 1 completes the proof by induction. 
\end{proof}

\begin{thm} \label{Sne3}
$S_n + \epsilon_3$ is minor minimal with respect to dimension $n+1$.
\end{thm}

\begin{proof}
That $\dim (S_n + \epsilon_3) = n+1$ follows immediately from \Cref{connector} and \Cref{Sn}. To show that $S_n + \epsilon_3$ is minor minimal with respect to dimension, we must show that every minor of $S_n + \epsilon_3$, $H$, has dimension less than $n+1$. We consider three exhaustive cases:
\begin{description}[style=unboxed]
\item[$H$ is obtained by only performing minor operations on $S_{n}$] In other words, $H = J + \epsilon_3$ for some $J \prec S_n$. Since $S_n$ is minor minimal with respect to spherical dimension $n-1$, $\sdim J < n-1$. Since $\epsilon_3$ can be embedded on a $\sphere[2]$ of any radius, \Cref{connector} yields that $\dim H < n +1$.
\item[$H$ is obtained by contracting at least one edge between $S_n$ and $\epsilon_3$ or removing at least one vertex from $\epsilon_3$] Then $H$ is a subgraph of $K_n + \epsilon_2$. Since $\sdim K_n = n-1$ and $\epsilon_2$ can be embedded on a $\sphere[1]$ of any radius, $\dim H < n+1$.
\item[$H$ is obtained by removing edges from between $S_{n-3}$ and $\epsilon_3$] It suffices to consider the case when only one edges is removed. Let $v$ and $u$ be the vertices incident to the removed edge, $v$ from $S_n$ and $u$ from $\epsilon_3$. Since $S_n$ is minor minimal with respect to spherical dimension, $\sdim (S_n - v) = n-2$. Set an embedding of $S_n - v$ on an $\sphere[n-2]$ of radius less than $\frac{\sqrt{2}}{2}$. Add $v$ back to expand this embedding of $S_n -v$ into an embedding of $S_n$ on an $\sphere[n-1]$. By \Cref{Main1}, we may add $\epsilon_2$ to the above construction to get $S_n + \epsilon_2$ in $\mathbb{R}^n$. \Cref{Main1} also yields that the points a distance 1 away from $S_n - v$ form a $\sphere[2]$. Place $u$ anywhere on this $\sphere[2]$ not occupied by $v$ or $\epsilon_2$ and add it to $S_n - v$ to complete the embedding of $H$ in $\mathbb{R}^n$. Thus, $\dim H < n+1$.
\end{description}
We have now shown that $\dim H < n+1$ for every $H \prec S_n + \epsilon_3$, so we are done.
\end{proof}


\section{And Another Class of Graphs Minor Minimal with Respect to Dimension: Flowers}

In this section, we add the additional requirement that edges do not cross in graph embeddings. We will use ``embedding" synonymously with ``embedding such that edges do not cross." The following proposition helps navigate this new restriction:

\begin{prop} \label{edges} 
Let $G$ and $H$ be graphs. Set an embedding of $G \cupdot H$ in $\mathbb{R}^{n+m}$ such that $G$ and $H$ lie in planes orthogonal to one another and every vertex of $G$ is a distance 1 from every vertex of $H$. Then adding every edge between $G$ and $H$ yields an embedding of $G+H$.
\end{prop}

\begin{proof}
We must show that the edges between $G$ and $H$ do not cross any other edges. To do this, construct $G+H$ as in the statement. Without loss of generality, assume that $G$ lies in the first n coordinates and $H$ lies in the last m coordinates (i.e., the vertices of $G$ lie on points of the form $(x_1,...,x_n,0,...,0)$ and the vertices of $H$ lie on points of the form $(0,...,0,x_{n+1},...,x_{n+m})$).

Since the edges within $G$ lie completely inside the surface $\{(x_1,...,x_n,0,...,0)\}$ and the edges between $G$ and $H$ lie completely outside this surface, these edges do not cross. The analogous statement also holds for the edges within $H$ and edges between $G$ and $H$. 

To complete the proof, we must show that two edges between $G$ and $H$ do not cross. Let $e=(u,v)$ and $e'=(u',v')$ be edges between $G$ and $H$, with $u, u'$ vertices of $G$ and $v,v'$ vertices of $H$. Either $u \neq u'$ or $v \neq v'$, so without loss of generality, we may assume $u \neq u'$. Indeed, \Cref{Main1} shows that $G$ lies on an n-dimensional sphere centered at the origin, so $u$ and $u'$ are linearly independent. Therefore, the projection of $e$ and $e'$ into the first $n$ dimensions do not cross, and so $e$ and $e'$ do not cross, completing the proof.
\end{proof}

Most everything up to this point still holds with the additional restriction that an embedding is legal only if edges do not cross. Indeed, the proofs are almost all identical except that we would have to cite the above theorem a few times. 

One notable exception is our work that relied on star n-gons. For example, $C_5$ cannot be embedded on a $\sphere[2]$ of radius less than $\frac{\sqrt{2}}{2}$ if we do not allow edges to cross. Therefore, using the argument at the very end of Section 3, we see that $\sdim (C_5 + K_2) = 5 > 3 = \sdim (C_5 + K_1)$.

\Cref{Final Wheel} is another example. For $n>6$, our embedding of $C_n$ on a $\sphere[2]$ was as a star n-gon, which has crossing edges. Therefore, $\sdim C_n = 3$ for $n \geq 6$ (by using the construction in \Cref{wheelFig}), and so we get the following analog to \Cref{Final Wheel}
\begin{prop}
If edges are not allowed to cross, the dimensions of wheel graphs are 
\begin{figure}[H]
\centering
\begin{tabular}{ c | c | c | c }
 \textbf{Dimension of $W_n^k$} & $3 \leq n <6$ & $n=6$ & $n > 6$ \\ \hline
$k=1$     & 3 & 2 & 3 \\
$k=2$     & 3 & 4 & 4 \\
$k\geq 3$ & 4 & 5 & 5 \\
 \end{tabular}
\end{figure}
\end{prop}

Now, notice that $W_6^2$ is minor minimal with respect to dimension 4 and $W_6^3$ is minor minimal with respect to dimension 5. In both cases, the graph almost fits in one fewer dimension. For example, if we put the two axle vertices of $W_6^2$ arbitrarily close together in $\mathbb{R}^3$, $C_6$ would have to fit on a circle of radius just less than 1, but $C_6$ can only fit on a circle of radius 1. Similarly, if we put the three axle vertices of $W_6^3$ on a circle of arbitrarily small radius in $\mathbb{R}^4$, $C_6$ would again have to fit on a circle of radius just less than 1. In both cases, adding the extra dimension allows us to embed $C_6$ on a $\sphere[3]$ (using the construction in \Cref{wheelFig}) and thus the wheel graph.

\begin{figure}[h]
\centering
\includegraphics[width=3cm]{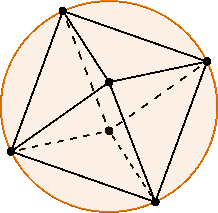} \hspace{0.5cm}
\includegraphics[width=3cm]{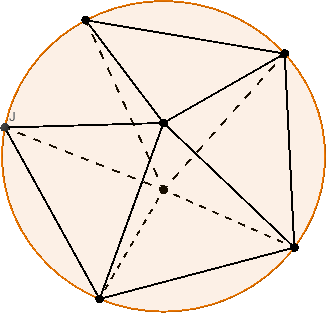}
\hspace{0.5cm}
\includegraphics[width=3.5cm]{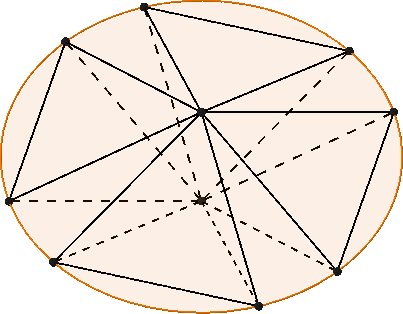}
\caption{A few graphs in $\mathcal{F}_4^2$. They can be obtained by splitting zero, one, or all the spokes of $W_4^2$, respectively.}
\label{example_flowers}
\end{figure}

Flower graphs are a generalization of wheel graphs. If we split an arbitrarily number of ``spokes" of the $W_6^2$, we get what we call the 6-Flowers with $\epsilon_2$ center, denoted $\mathcal{F}_6^{\epsilon_2}$. More precisely, we define petals by
\[
\mathcal{P}_n = \left \{ P \suchthat \begin{aligned} &\text{$P$ has $n$ edges, and each vertex of $P$} \\
&\text{ is incident to either one or two edges} \end{aligned} \right \}
\]
and flowers by
\[
\mathcal{F}_n^G = \{ G + P \suchthat P \in \mathcal{P}_n\}.
\]

The following is the main theorem of this section:
\begin{thm} \label{flowers}
The following flower graphs are minor minimal:
\begin{itemize}
	\item $\mathcal{F}_6^{S_1} - \{ W_6 \}$ with respect to dimension 3,
	\item $\mathcal{F}_6^{S_n}$ ($1<n<5$) with respect to dimension $n+2$,
	\item $\mathcal{F}_5^{S_n}$ ($n \geq 5$) with respect to dimension $n+2$.
\end{itemize}
\end{thm}

Before we prove this theorem, we need to state and prove some lemmas. 

\begin{lem} \label{radiiCor}
Let $r_n$ be the circum-radius of the regular convex n-gon.
\begin{description}
    \item[a] For all $n$, $S_n$ can be embedded on an $\sphere[n-1]$ of radius less than $\sqrt{1-r_4^2}$.
    \item[b] For $1 \leq n \leq 4$, $S_n$ can be embedded on a $\sphere[3]$ of radius less than $\sqrt{1-r_5^2}$.
    \item[c] For $n>4$, $S_n$ cannot be embedded on an $\sphere[n-1]$ of radius less than or equal to $\sqrt{1-r_5^2}$.
\end{description}
\end{lem}

\begin{proof}
To prove the three claims, we simply need to do some computations. \Cref{radii} will do most of the work for us.
\\

\noindent \textbf{a.} Direct computation yields $r_4 = \frac{\sqrt{2}}{2}$, and  \Cref{radii} and \Cref{Main2Cor}a yield that $S_n$ can be embedded on an $\sphere[n-1]$ of radius less than $\frac{\sqrt{2}}{2} = \sqrt{1-r_4^2}$.
\\

\noindent \textbf{b.} Direct computation yields $1 - r_5^2 = \frac{5+\sqrt{5}}{10}$, and \Cref{radii} yields that $S_n$ can be embedded on a $\sphere[3]$ of radius arbitrarily close to $\frac{1}{2} < \sqrt{1-r_5^2}$.
\\

\noindent \textbf{c.} By \Cref{radii}, $S_5$ cannot be embedded on a $\sphere[4]$ of radius less than $R \left ( \frac{1}{2} \right ) > \sqrt{1-r_5^2}$. Thus, \Cref{RProps}a shows that $S_n$ cannot be embedded on an $\sphere[n-1]$ of radius less than $\sqrt{1-r_5^2}$ for all $n > 4$.
\end{proof}

\begin{lem} \label{Sminors}
Every proper minor of $S_n$ can be embedded on an $\sphere[n-2]$ of radius $\frac{\sqrt{2}}{2}$.
\end{lem}

\begin{proof}
The proper minors of $S_n$ are all subgraphs of $K_{n-4} + (P_3 \cupdot vertex)$ or $K_{n-5} + K_{3,2}$. Therefore, it suffices to show that both of these graphs can be embedded on an $\sphere[n-2]$ of radius $\frac{\sqrt{2}}{2}$.

In both cases, we will write the graph as the sum of two graphs, $G+H$, and show that $G$ can be embedded on an $\sphere[m]$ and $H$ on an $\sphere[m']$ such that $m + m' = n-2$. This is sufficient because we can then embed $G$ on an $\sphere[m]$ and $H$ on an $\sphere[m']$ that have the same center but lie in orthogonal planes. Adding all edges between $G$ and $H$ will then yield an embedding of $G+H$ on an $\sphere[n-2]$ of radius $\frac{\sqrt{2}}{2}$.

Since $K_{n-4}$ can be embedded on an $\sphere[n-5]$ of radius less than $\frac{\sqrt{2}}{2}$ (\Cref{KSphere}), it can be embedded on an $\sphere[n-4]$ of radius $\frac{\sqrt{2}}{2}$. It is easy too see that $P_3 \cupdot vertex$ can be embedded on a $\sphere[2]$ of radius $\frac{\sqrt{2}}{2}$. 

Since $K_{n-5}$ can be embedded on an $\sphere[n-6]$ of radius less than $\frac{\sqrt{2}}{2}$, it can be embedded on an $\sphere[n-5]$ of radius $\frac{\sqrt{2}}{2}$. Additionally, $K_{3,2}$ can be embedded on a $\sphere[3]$ of radius equal to $\frac{\sqrt{2}}{2}$.
\end{proof}

\begin{lem} \label{petal}
For $3 \leq n\leq 6$, $P \in \mathcal{P}_n$ can be embedded on circle (but obviously not two points). Moreover, $P$ can be embedded on all circles of radius greater than $r_n$ but cannot be embedded on any circle of radius less than $r_n$.
\end{lem}

\begin{proof}
This is straight-forward after recalling the definitions of $R_P$ and $r_n$ (which appear directly before \Cref{4verts} and in \Cref{radiiCor}, respectively).
\end{proof}

\begin{lem} \label{petal minors}
For $4 \leq n \leq 6$, suppose some minor operation is performed on $P \in \mathcal{P}_n$ to get $P'$. Then $R_{P'} \leq r_{n-1}$.
\end{lem}

\begin{proof}
The main idea is to use that $R_{P'} \leq r_{n-1}$ if $P' \in \mathcal{P}_{n-1}$. We break the proof into two exhaustive cases:
\begin{description}
\item[$P'$ does not have a disjoint vertex] Then $P' \in \mathcal{P}_{n-1}$, and so we are done.
\item[$P'$ has a disjoint vertex] Let $v$ be the disjoint vertex. Then $P'-v \in \mathcal{P}_{n-1}$, and so we can complete the embedding of $P'$ by placing $v$ on any unoccupied spot on the circle.
\end{description}
\end{proof}

If $P \in \mathcal{P}_n - \{C_n\}$, we can strengthen the above lemma:

\begin{lem} \label{petal minors cor}
Let $4 \leq n \leq 6$. If $P'$ is a minor of $P \in \mathcal{P}_n - \{C_n\}$, then $P'$ can be embedded on a circle, $S$, of radius arbitrarily close to $r_{n-1}$ such that a straight line can be drawn from the center of $S$ to any vertex of $P'$ without crossing an edge. 
\end{lem}

\begin{proof}
The proof is almost identical to that of \Cref{petal minors}. We need only slightly revise the second case:
\begin{description}
\item[$P'$ has a disjoint vertex] Proceed as before, but place $v$ on any unoccupied spot on $S$ such that a straight line can be drawn from the center of $S$ to $v$ without crossing an edge of $P'$.
\end{description}
\end{proof}

We now prove \Cref{flowers}:

\begin{pf}{\Cref{flowers}} \\
\textbf{a.} Recall that $\dim W_6 = 2$, so $W_6$ is not minor minimal with respect to dimension 3. Let $F_6 \in \mathcal{F}_6^{S_1}-\{W_6\}$, and let $P$ be $F_6$'s petals. $F_6$ can be embedded in $\mathbb{R}^m$ if and only if $P$ can be embedded on an unit m-dimensional sphere. Since $P$ cannot be embedded on a unit circle, $\dim F_6 \geq 3$. However, $P$ can be embedded on a unit sphere (as demonstrated in \Cref{wheelFig}), so $\dim F_6 = 3$.

We must now show that $\dim H < 3$ for every $H \prec F_6$. We consider a few exhaustive cases:
\begin{description}[style=unboxed]
    \item[$S_1$ is deleted] Then $H$ is a subgraph of $P$. Since $\dim P_6 = 1$, $\dim H \leq 1 < 3$.
    \item[A minor operation is performed on $P$, or an edge between $S_1$ and $P_6$ is contracted] In either case, the resulting graph is a subgraph of $S_1 + P'$, where $P' \prec P$. The result now follows from \Cref{petal minors cor}.
    \item[An edge between $P$ and $S_1$ is deleted] \Cref{minors of F_6} suggests the proof. 
    \begin{figure}[h]
    \centering
    \includegraphics[width=3cm]{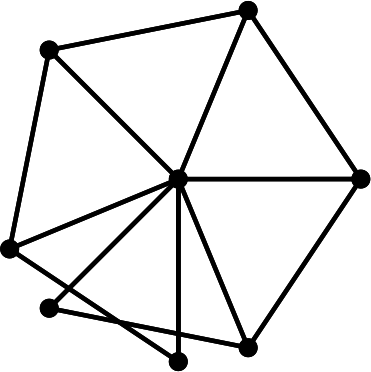} 
    \hspace{0.5cm}
    \includegraphics[width=3cm]{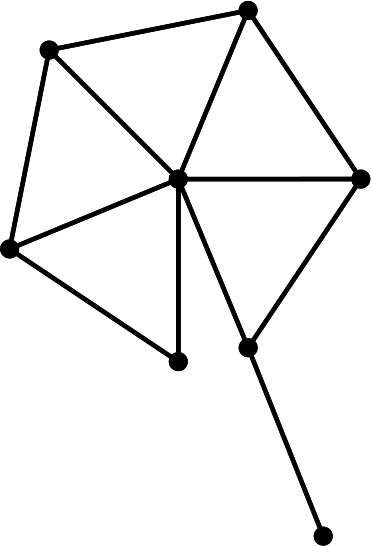}
    \hspace{0.5cm}
    \includegraphics[width=3cm]{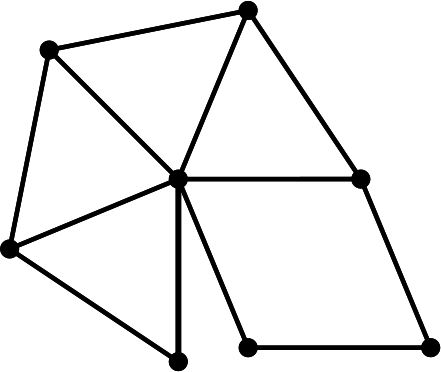} 
    \caption{$F_6$ cannot be embedded in $\mathbb{R}^2$, as can be seen in the left-most graph. However, if we remove an edge between a petal vertex and the central vertex, we get a graph akin to one of the two right-most graphs, which can be embedded $\mathbb{R}^2$.}
    \label{minors of F_6}
    \end{figure}
\end{description}
Therefore, regardless, $\dim H < 3$, as desired.
\\
\\
\noindent \textbf{b.} Let $F_6 \in \mathcal{F}_6^{S_n}$ for some $1 < n < 5$, and let $P$ be its petals. Notice that $P$ has spherical dimension 3 and  can be embedded on a $\sphere[3]$ of any radius greater than $\frac{1}{2}$. This, in tandem with \Cref{radiiCor}a and \Cref{connector}, shows that $\dim F_6 = n+2$. 

Now, let $H \prec F_6$. We must show that $\dim H < n+2$. We consider a few exhaustive cases:
\begin{description}[style=unboxed]
    \item[Minor operations are performed on $S_n$, but no edges are contracted between $S_n$ and $P$] From \Cref{Sminors}, every minor of $S_n$ can be embedded on an $\sphere[n-2]$ of radius $\frac{\sqrt{2}}{2}$. Since $P$ can be embedded on a $\sphere[3]$ of radius $\frac{\sqrt{2}}{2}$, \Cref{connector} yields that $\dim H = n+1 < n+2$.
    \item[A minor operation is performed on $P$, or an edge between $S_n$ and $P$ is contracted] In either case, the resulting graph is a subgraph of $K_n + P'$, where $P' \prec P$. Using \Cref{KSphereImp}, we see that $K_n$ can be embedded on an $\sphere[n-1]$ of radius less than $\sqrt{1-r_5^2}$, and by \Cref{petal minors}, $P'$ can be embedded on a $\sphere[2]$ of any radius greater than $r_5$. Thus, \Cref{connector} yields that $\dim H = n+1 < n+2$.
    \item[An edge between $S_n$ and $P$ is deleted] \Cref{radiiCor}b together with \Cref{minors of F_6} suggests the proof.
\end{description}

\noindent \textbf{c.} Let $F_5 \in \mathcal{F}_5^{S_n}$ for some $n \geq 5$, and let $P$ be its petals. Notice that $P$ can be embedded on a $\sphere[3]$ of any radius greater than $\frac{1}{2}$. This in tandem with \Cref{radiiCor}a shows that $F_5$ can be embedded in $\mathbb{R}^{n+2}$. Moreover, \Cref{petal}, \Cref{radiiCor}c, and \Cref{connector} shows that $\dim F_5 > n+1$. Therefore, $\dim F_5 = n+1$>

Now, let $H \prec F_5$. We must show that $\dim H < n+2$. We consider a few exhaustive cases.
\begin{description}[style=unboxed]
    \item[Minor operations are performed on $S_n$, but no edges are contracted between $S_n$ and $P$] From \Cref{Sminors}, every minor of $S_n$ can be embedded on an $\sphere[n-2]$ of radius $\frac{\sqrt{2}}{2}$. Since $P$ can be embedded on a $\sphere[3]$ of radius $\frac{\sqrt{2}}{2}$, \Cref{connector} yields that $\dim H = n+1 < n+2$.
    \item[A minor operation is performed on $P$, or an edge between $S_n$ and $P$ is contracted] In either case, the resulting graph is a subgraph of $K_n + P'$, where $P' \prec P$. By \Cref{KSphere}, $K_n$ can be embedded on an $\sphere[n-1]$ of radius less than $\frac{\sqrt{2}}{2}$, and by \Cref{petal minors}, $P'$ can be embedded on a $\sphere[2]$ of any radius greater than $\frac{\sqrt{2}}{2}$. Thus, \Cref{connector} yields that $\dim H = n+1 < n+2$.
    \item[An edge between $S_n$ and $P$ is deleted] \Cref{radiiCor}a together with \Cref{minors of F_6} suggests the proof.
\end{description}
\end{pf}

Notice that \Cref{flowers} does not hold with our previous notion of dimension. This is because we could then stack petals on top of one another as in the left-most graph in \Cref{minors of F_6}. Therefore, most of the flower graphs we have considered would have dimension $n+1$ rather than $n+2$. However, this stacking method does not work if the petals are a cycle graph. Therefore, $S_n + C_6$ still has dimension $n+2$ for $1<n<5$ and so is still minor minimal. Similarly $S_n + C_5$ is still minor minimal for $n\geq 5$. This is our final result:

\begin{cor} \label{final theorem}
Using the definition of dimension used in the previous sections, the following flower graphs are minor minimal with respect to dimension $n+2$:
\begin{itemize}
    \item $S_n + C_6$ for $1<n<5$,
    \item $S_n + C_5$ for $n\geq 5$.
\end{itemize}
\end{cor}

\section{Acknowledgments}

The first author would like to thank co-advisor Dr. Harold Ellingsen for his patient and prudent mentorship. He would also like to thank all of his colleagues from the Potsdam REU, especially Ioherase Ransom and Jessica Mean. Their engagement in speculative conversations helped make this paper possible.

This research was conducted through the SUNY Potsdam/Clarkson University REU, with funding from the National Science Foundation under Grant No. DMA-1262737 and the National Security Administration under Grant No. H98230-14-1-0141.

\end{document}